\documentclass{ws-book9x6}
\usepackage{ws-book-van}   
\usepackage{hyperref}  
\usepackage{amsmath,amssymb,amsthm}
\usepackage{centernot}
\usepackage{soul}
\usepackage{tikz}
\usetikzlibrary{arrows}
\usetikzlibrary{bending}

\begin{document}


%

\newcommand{\fq}{\mathbb{F}_q}
\newcommand{\diam}{{\rm diam} }
\newcommand{\dist}{{\rm dist} }

\theoremstyle{plain}
\newtheorem{theorem}{Theorem}[section]
\newtheorem{proposition}{Proposition}[section]
\newtheorem{lemma}{Lemma}[section]
\newtheorem{cor}{Corollary}[section]
\theoremstyle{definition}
\newtheorem{example}[theorem]{Example}
\newtheorem{remark}{Remark}
\newtheorem{conjecture}{Conjecture}[section]

\chapter[Diameter of some monomial digraphs]
{
Diameter of some monomial digraphs.
}

\author{A. Kodess}
\address{University of Rhode Island, Kingston, RI, USA \url{kodess@uri.edu}}

\author{F. Lazebnik}
\address{University of Delaware, Newark, DE, USA \url{fellaz@udel.edu}}

\author{S. Smith}
\address{University of Delaware, Newark, DE, USA \url{smithsj@udel.edu}}

\author{J. Sporre}
\address{University of Delaware, Newark, DE, USA \url{jsporre@udel.edu}}

\section{Introduction}
For all  terms related to digraphs which are not defined below, see Bang-Jensen and Gutin \cite{Bang_Jensen_Gutin}.
In this paper,
by a {\it directed graph} (or simply {\it digraph)}
$D$ we mean a pair $(V,A)$, where
$V=V(D)$ is the set of vertices and $A=A(D)\subseteq V\times V$ is the set of arcs.
For an arc $(u,v)$, the first vertex $u$ is called its {\it tail} and the second
vertex $v$ is called its {\it head}; we also denote such an arc by $u\to v$.
If $(u,v)$ is an arc,  we call $v$  an {\it out-neighbor}  of $u$, and $u$ an {\it in-neighbor} of $v$.
The number of out-neighbors of  $u$ is called the {\it out-degree} of $u$, and the number of in-neighbors of $u$ --- the {\it in-degree} of $u$.
For an integer $k\ge 2$,  a {\it walk} $W$ {\it from} $x_1$ {\it to} $x_k$ in $D$ is an alternating sequence
$W = x_1 a_1 x_2 a_2 x_3\dots x_{k-1}a_{k-1}x_k$ of vertices $x_i\in V$ and arcs $a_j\in A$
such that the tail of $a_i$ is $x_i$ and the head of $a_i$ is $x_{i+1}$ for every
$i$, $1\le i\le k-1$.
Whenever the labels of the arcs of a walk are not important, we use the notation
$x_1\to x_2 \to \dotsb \to x_k$ for the walk, and say that we have an $x_1x_k$-walk.
In a digraph $D$, a vertex $y$ is {\it reachable} from a vertex $x$ if $D$ has a walk from $x$ to $y$. In
particular, a vertex is reachable from itself. A digraph $D$ is {\it strongly connected}
(or, just {\it strong}) if, for every pair $x,y$ of distinct vertices in $D$,
$y$ is reachable from $x$ and $x$ is reachable from $y$.
A {\it strong component} of a digraph $D$ is a maximal induced subdigraph of $D$ that is strong.
If $x$ and $y$ are vertices of a digraph $D$, then the
{\it distance from x to y} in $D$, denoted $\dist(x,y)$, is the minimum length of
an $xy$-walk, if $y$ is reachable from $x$, and otherwise $\dist(x,y) = \infty$.
The {\it distance from a set $X$ to a set $Y$} of vertices in $D$ is
\[
\dist(X,Y) = \max
\{
\dist(x,y) \colon x\in X,y\in Y
\}.
\]
The {\it diameter} of $D$ is $\diam(D) = \dist(V,V)$.


Let $p$ be a prime, $e$ a positive integer, and $q = p^e$. Let
$\fq$ denote the finite field of $q$ elements, and  $\fq^*=\fq\setminus\{0\}$.

Let $\fq^2$ 
denote the Cartesian product $\fq \times \fq$, and let
 $f\colon\fq^2\to\fq$ be an arbitrary function. We define a  digraph $D = D(q;f)$  as follows:
 $V(D)=\fq^{2}$, and
there is an arc from a vertex ${\bf x} = (x_1,x_2)$ to a vertex
${\bf y} = (y_1,y_{2})$ if and only if
\[
x_2 + y_2 = f(x_1,y_1).
\]

If $(x,y)$ is an arc in $D$, then ${\bf y}$ is uniquely determined by  ${\bf x}$ and $y_1$, and ${\bf x}$ is uniquely determined by  ${\bf y}$ and $x_1$.
Hence,  each vertex of $D$ has both its in-degree and out-degree equal to $q$.

By Lagrange's interpolation,
 $f$ can be uniquely represented by
a bivariate polynomial of degree at most $q-1$ in each of the variables.  If  ${f}(x,y) =  x^m y^n$, $1\le m,n\le q-1$, we call  $D$ a {\it monomial} digraph, and denote it also by  $D(q;m,n)$. Digraph $D(3; 1,2)$ is depicted in Fig.\ $1.1$. It is clear,  that ${\bf x}\to {\bf y}$ in $D(q;m,n)$ if and only if ${\bf y}\to {\bf x}$ in $D(q;n,m)$. Hence, one digraph is obtained from the other by reversing the direction of every arc. In general, these digraphs are not isomorphic, but if one of them is strong then so is the other  and their diameters are equal. As this paper is concerned only with the diameter of $D(q;m,n)$,  it is sufficient to assume that $1\le m\le n\le q-1$.

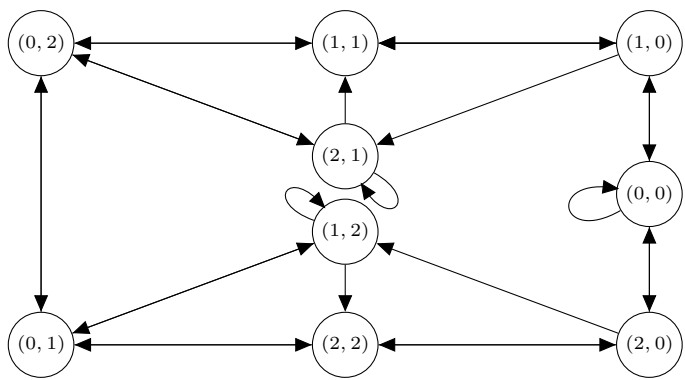
\begin{figure}
\begin{center}
\begin{tikzpicture}

\tikzset{vertex/.style = {shape=circle,draw,inner sep=2pt,minimum size=.5em, scale = 1.0},font=\sffamily\scriptsize\bfseries}
\tikzset{edge/.style = {->,> = triangle 45}}
\node[vertex] (a) at  (0,0) {$(0,2)$};
\node[vertex] (b) at  (4,0) {$(1,1)$};
\node[vertex] (c) at  (8,0) {$(1,0)$};
\node[vertex] (d) at  (0,-4) {$(0,1)$};
\node[vertex] (e) at  (4,-4) {$(2,2)$};
\node[vertex] (f) at  (8,-4) {$(2,0)$};
\node[vertex] (g) at  (4,-1.5) {$(2,1)$};
\node[vertex] (h) at  (4,-2.5) {$(1,2)$};
\node[vertex] (i) at  (8,-2) {$(0,0)$};

\draw[edge] (a) to (b);
\draw[edge] (b) to (a);

\draw[edge] (a) to (d);
\draw[edge] (d) to (a);

\draw[edge] (b) to (c);
\draw[edge] (c) to (b);

\draw[edge] (g) to (b);

\draw[edge] (h) to (e);
\draw[edge] (c) to (b);

\draw[edge] (d) to (e);
\draw[edge] (e) to (d);

\draw[edge] (e) to (f);
\draw[edge] (f) to (e);

\draw[edge] (c) to (i);
\draw[edge] (i) to (c);

\draw[edge] (f) to (i);
\draw[edge] (i) to (f);

\draw[edge] (g) to (a);
\draw[edge] (a) to (g);

\draw[edge] (c) to (g);

\draw[edge] (d) to (h);
\draw[edge] (h) to (d);

\draw[edge] (f) to (h);

\path
    (g) edge [->,>={triangle 45[flex,sep=-1pt]},loop,out=330,in=300,looseness=8] node {} (g);
\path
    (h) edge [->,>={triangle 45[flex,sep=-1pt]},loop,out=160,in=130,looseness=8] node {} (h);
\path
    (i) edge [->,>={triangle 45[flex,sep=-1pt]},loop,out=210,in=170,looseness=8] node {} (i);
\end{tikzpicture}
\caption{The digraph $D(3;1,2)$: $x_2+y_2 = x_1y_1^2$.}
\end{center}
\end{figure}
%
%

The digraphs $D(q; {f})$
and $D(q;m,n)$ are directed analogues
of
some algebraically defined graphs,  which have been studied extensively
and have many applications. See
Lazebnik and Woldar \cite{LazWol01} and references therein; for some
subsequent work see  Viglione \cite{Viglione_thesis},
Lazebnik and Mubayi \cite{Lazebnik_Mubayi},
Lazebnik and Viglione \cite{Lazebnik_Viglione},
Lazebnik and Verstra\"ete \cite{Lazebnik_Verstraete},
Lazebnik and Thomason \cite{Lazebnik_Thomason},
 Dmytrenko, Lazebnik and Viglione \cite{DLV05},
 Dmytrenko, Lazebnik and Williford \cite{DLW07},
 Ustimenko \cite{Ust07}, Viglione \cite{VigDiam08},
 Terlep and Williford \cite{TerWil12}, Kronenthal \cite{Kron12},
 Cioab\u{a}, Lazebnik and Li \cite{CLL14},
 Kodess \cite{Kod14},
and  Kodess and Lazebnik \cite{Kod_Laz_15}.

The questions of strong connectivity of digraphs $D(q;{f})$ and $D(q; m,n)$ and descriptions of their components were completely answered in
\cite{Kod_Laz_15}.  Determining the diameter of a component of $D(q;{f})$ for an arbitrary prime power $q$ and an arbitrary $f$ seems to be out of reach,  and most of our results below are concerned with some instances of this problem for strong monomial digraphs.  The following theorems are the main results of this paper.

\begin{theorem}
\label{main}
Let $p$  be a prime,  $e,m,n$ be  positive integers, $q=p^e$, $1\le m\le n\le q-1$,  and  $D_q= D(q;m,n)$. Then the following statements hold.
\begin{enumerate}
\item\label{gen_lower_bound}  If $D_q$ is strong, then $\diam (D_q)\ge 3$.

\item\label{gen_upper_bound}
If $D_q$ is strong, then
\begin{itemize}
\item for $e = 2$,   $\diam(D_q)\le 96\sqrt{n+1}+1$;
\item for $e \ge 3$, $\diam(D_q)\le 60\sqrt{n+1}+1$.
\end{itemize}

\item\label{diam_le_4} If $\gcd(m,q-1)=1$ or $\gcd(n,q-1)=1$, then $\diam(D_q)\le 4$.
If $\gcd(m,q-1) = \gcd(n,q-1) = 1$, then $\diam(D_q) = 3$.

\item \label{main3}  If $p$ does not divide $n$, and $q > (n^2-n+1)^2$,
then $\diam(D(q;1,n)) = 3$.
\item If $D_q$ is strong, then:
\begin{enumerate}
\item[(a)\label{bound_q_le25}]
 If $q > n^2$, then $\diam(D_q) \le 49$.  
\item[(b)\label{bound_q_m4n4}]
If $q > (m-1)^4$, then $\diam(D_q)\le 13$.
\item[(c)]\label{bound_q_le6} If $q > (n-1)^4$, then $\diam(D(q;n,n))\le 9$.
\end{enumerate}
\end{enumerate}
\end{theorem}
\begin{remark}
The converse to either of the statements in part (\ref{diam_le_4}) of Theorem \ref{main} is not true. Consider, for instance,
$D(9;2,2)$ of diameter $4$, or $D(29;7,12)$ of diameter $3$.
\end{remark}
\begin{remark}
The result of part \ref{bound_q_le25}a can hold for some $q\le m^2$.
\end{remark}
For prime $q$, some of the results of Theorem \ref{main} can be strengthened.

\begin{theorem}
\label{thm_diam_p}
Let $p$  be a prime, $1\le m \le n\le p-1$, and  $D_p= D(p;m,n)$. Then $D_p$ is strong and the following statements hold.
\begin{enumerate}
\item\label{diam_bound_p}
$\diam (D_p) \le 2p-1$ with equality if
and only if
$m=n=p-1$.
\item\label{bound_p_sqrt60}
If $(m,n)\not\in\{((p-1)/2,(p-1)/2),((p-1)/2,p-1), (p-1,p-1)\}$,
 then $\diam(D_p)\le 120\sqrt{m}+1$.
\item\label{bound_p_le10}
If $p > (m-1)^3$,
            then $\diam(D_p) \le 19$.  
\end{enumerate}
\end{theorem}

The paper is organized as follows. In section \ref{preres} we present all results which are needed for our proofs of Theorems \ref{main} and \ref{thm_diam_p} in sections \ref{proofs1} and \ref{proofs2}, respectively.  Section \ref{open} contains concluding remarks and open problems.

\section{Preliminary results.}\label{preres}
We begin with a general result that gives necessary and sufficient conditions for a digraph $D(q;m,n)$ to be strong.


\begin{theorem} {\rm [\cite{Kod_Laz_15}, Theorem 2]}
\label{thm_conn}
$D(q;m,n)$ is strong if and only if $\gcd(q-1,m,n)$ is not divisible by any
$q_d = (q-1)/(p^{d}-1)$ for any positive divisor $d$ of $e$, $d < e$.
In particular, $D(p;m,n)$ is strong for any $m,n$.
\end{theorem}

Every walk of length $k$ in $D = D(q; m,n)$ originating at $(a,{b})$ is of the form
\begin{align}
(a, b)    &\to (x_1,- b + a^m x_1^n)\nonumber\\
          &\to (x_2, b - a^m x_1^n + x_1^m x_2^n)\nonumber\\
          &\to \cdots \nonumber\\
          &\to(x_k, x_{k-1}^m x_k^n- x_{k-2}^m x_{k-1}^n+\cdots +(-1)^{k-1} a^m x_1^n+(-1)^k b)\nonumber.
\end{align}

Therefore, in order to prove that $\diam(D)\le k$, one can  show that  for any choice of $a,b,u,v\in\fq$, there  exists $(x_1,\dotso,x_k)\in\fq^k$ so that
\begin{equation}
\label{eqn:walk_length_k}
(u,v) = (x_k, x_{k-1}^m x_k^n- \cdots +(-1)^{k-1} a^m x_1^n+(-1)^k b).
\end{equation}
In order to show that $\diam(D)\ge  l$, one can show that there exist $a,b,u,v\in~\fq$ such that
(\ref{eqn:walk_length_k}) has no solution in $\fq^k$ for any $k < l$.
\bigskip


\subsection{
Waring's Problem
}
In order to obtain an upper  bound on $\diam(D(q; m,n))$ we will use some results concerning Waring's problem over finite fields.

Waring's number $\gamma(r,q)$ over $\fq$ is defined as the smallest positive integer $s$ (should it exist) such that the equation
\[
x_1^r + x_2^r + \dotsb + x_s^r = a
\]
has a solution $(x_1,\dotso,x_s)\in\fq^s$ for any $a\in\fq$.
Similarly, $\delta(r,q)$ is defined as the smallest positive integer $s$ (should it exist) such that
for any $a\in\fq$, there exists $(\epsilon_1,\dotso,\epsilon_s)$,
each $\epsilon_i\in\{-1,1\}\subseteq\mathbb{F}_q$,
for which the equation
\[
\epsilon_1 x_1^r + \epsilon_2 x_2^r + \dotsb + \epsilon_s x_s^r = a
\]
 has a solution $(x_1,\dotso,x_s)\in\fq^s$.
It is easy to argue that  $\delta(r,q)$ exists if and only if
$\gamma(r,q)$ exists, and in this case $\delta(r,q)\le \gamma(r,q)$.

A criterion on the existence of $\gamma(r,q)$ is the following theorem  by Bhashkaran \cite{Bhashkaran_1966}.

\begin{theorem} {\rm [\cite{Bhashkaran_1966}, Theorem G]}
\label{thm:waring_exist}
Waring's number $\gamma(r,q)$ exists if and only if $r$ is not divisible by any $q_d
 = (q-1)/(p^{d}-1)$ for any positive divisor $d$ of $e$, $d < e$.
\end{theorem}
The study of various bounds on $\gamma(r,q)$ has drawn considerable attention. We will use the following two upper bounds on Waring's number due to J.~Cipra \cite{Cipra_2009}.
\begin{theorem}{\rm [\cite{Cipra_2009}, Theorem 4]}
\label{thm:waring_bound}
If $e = 2$ and $\gamma(r,q)$ exists,
then $\gamma(r,q)\le 16\sqrt{r+1}$. Also, if
$e \ge 3$ and $\gamma(r,q)$ exists,
then $\gamma(r,q)\le 10\sqrt{r+1}$.
\end{theorem}
\begin{cor} {\rm [\cite{Cipra_2009}, Corollary 7]}
\label{thm:diam_le_8}
If $\gamma(r,q)$ exists and $r < \sqrt{q}$, then $\gamma(r,q)\le 8$.
\end{cor}
For the case $q = p$, the following bound will be of interest.
\begin{theorem}{\rm [Cochrane, Pinner \cite{Cochrane_Pinner_2008}, Corollary 10.3]}
\label{thm:Cochrane_Pinner}
If $|\{x^k\colon x\in\mathbb{F}_p^\ast\}|>2$, then $\delta(k,p)\le 20\sqrt{k}$.
\end{theorem}

The next two statements concerning very strong bounds on Waring's number in large fields follow from the work of Weil \cite{Weil}, and Hua and Vandiver \cite{Hua_Vandiver}.
\begin{theorem}{\rm [Small \cite{Small_1977}]}
\label{thm:waring_Small_estimates}
If $q > (k-1)^4$, then $\gamma(k,q) \le 2$.
\end{theorem}
\begin{theorem} {\rm [Cipra \cite{Cipra_thesis}, p.~4]}
\label{thm:waring_small_estimates}
If $ p > (k-1)^3$, then $\gamma(k,p)\le 3$.
\end{theorem}

For a survey on Waring's number over finite fields, see Castro and Rubio (Section 7.3.4, p.~211),
and Ostafe and Winterhof (Section 6.3.2.3, p.~175)
in Mullen and Panario \cite{Handbook2013}.  See also Cipra \cite{Cipra_thesis}.

We will need the following technical lemma.
\begin{lemma}
	\label{lemma:alt}
	Let $\delta = \delta(r,q)$ exist, and $k \ge 2\delta$.
	Then for every $a\in\fq$ the equation
	\begin{equation}
	\label{eqn:lemma_alt}
	x_1^r - x_2^r + x_3^r - \dotsb + (-1)^{k+1} x_k^r = a
	\end{equation}
	has a solution $(x_1,\dotso,x_k)\in\fq^k$.
\end{lemma}
\begin{proof}
	Let $a\in\fq$ be arbitrary. There exist $\varepsilon_1,\dotso,\varepsilon_\delta$, each
	$\varepsilon_i\in\{-1,1\}\subseteq \fq$, such that
	the equation
	$\sum_{i=1}^{\delta} \varepsilon_i y_i^r = a$ has a solution
	$(y_1,\dotso,y_{\delta})\in\fq^{\delta}$.
	As  $k \ge 2\delta$, the alternating sequence
	$1,-1,1,\dotso,(-1)^k$ with $k$ terms contains the sequence
	$\varepsilon_1,\dotso,\varepsilon_\delta$ as a subsequence.
	Let the indices of this subsequence be
	$j_1,j_2,\dotso,j_{\delta}$.
	For each $l$, $1\le l\le k$, let
	$x_l = 0$ if $l\neq j_i$ for any $i$, and
	$x_l = y_i$ for $l = j_i$. Then $(x_1,\dotso,x_k)$ is a solution of
	(\ref{eqn:lemma_alt}).
\end{proof}

\subsection{The Hasse-Weil bound}
In the next section we will use
the Hasse-Weil bound,
which  provides
a bound on the number of $\fq$-points on a plane non-singular absolutely irreducible projective curve over a finite field $\fq$.
If the number of points on the curve $C$ of genus $g$ over the
finite field $\fq$ is $|C(\fq)|$, then
\begin{equation}
\label{hasse_weil_bound}
||C(\fq)| - q -1|
\le
2g\sqrt{q}.
\end{equation}
  It is also known that for a non-singular curve
  defined by a homogeneous polynomial  of degree $k$, $g= (k-1)(k-2)/2$. Discussion of all related notions and a proof of this result can be found in
  Hirschfeld, Korchm\'{a}ros, Torres \cite{Hirschfeld} (Theorem 9.18, p.~343) or in Sz\H{o}nyi \cite{Szonyi1997} (p.~197).

\section{Proof of Theorem \ref{main}} \label{proofs1}

\noindent {\bf (\ref{gen_lower_bound}).}
As there is a loop at $(0,0)$, and there are arcs between  $(0,0)$ and $(x,0)$ in either direction, for every $x\in \fq^*$, the number of vertices in $D_q$ which are at distance at most 2 from $(0,0)$ is
at most $1+ (q-1)+(q-1)^2 < q^2$. Thus, there are vertices in $D_q$ which are at distance
at least 3 from $(0,0)$, and so $\diam(D_q)\ge 3$.

\bigskip

\noindent {\bf (\ref{gen_upper_bound}).}
As $D_q$ is strong, by Theorem \ref{thm_conn},
for any positive divisor $d$ of $e$, $d<e$,
$q_d\centernot\mid\gcd (p^e-1, m,n)$. As, clearly, $q_d\,|\,(p^e-1)$, either $q_d\centernot\mid m$ or $q_d\centernot\mid n$.  This implies by Theorem \ref{thm:waring_exist} that either $\gamma(m,q)$ or $\gamma(n,q)$ exists.

Let $(a,b)$ and $(u,v)$ be arbitrary vertices of $D_q$.  By (\ref{eqn:walk_length_k}), there exists a walk  of length at most $k$ from $(a,b)$ to $(u,v)$ if the  equation
\begin{equation}
\label{eqn:main}
v = x_{k-1}^m u^n- x_{k-2}^m x_{k-1}^n+\cdots +(-1)^{k-1} a^m
 x_1^n+(-1)^k b
\end{equation}
has a solution $(x_1,\ldots, x_k)\in \fq^k$.

Assume first that $\gamma_m = \gamma(m,q)$ exists.
Taking $k=6\gamma_m + 1$,
and $x_i = 0$ for $i\equiv 1 \mod 3$, and $x_i = 1$ for $i\equiv 0\mod 3$, we have that (\ref{eqn:main}) is equivalent to
\[
-x_{k-2}^m + x_{k-5}^m -\cdots +(-1)^k x_5^m + (-1)^{k-1}x_2^m  = v-(-1)^k b-u^n.
\]
As the number of terms on the left is $(k-1)/3 = 2 \gamma_m$, this equation  has a solution in $\fq^{2\gamma_m}$ by Lemma \ref{lemma:alt}.
Hence, (\ref{eqn:main}) has a solution in $\fq^{k}$.

If $\gamma_n = \gamma(n,q)$ exists, then the argument is similar: take $k = 6\gamma_n+1$,  $x_i = 0$ for $i\equiv 0 \mod 3$, and $x_i = 1$ for $i\equiv 1\mod 3$.

The result now follows from the bounds on $\gamma(r,q)$ in Theorem \ref{thm:waring_bound}.

\begin{remark}
	As $m\le n$, if $\gamma(m,q)$ exists, the upper bounds in Theorem~\ref{main},
	part {\bf (\ref{gen_upper_bound})}, can be improved by replacing $n$ by $m$. Also, if a better upper bound on $\delta(m,q)$ than $\gamma(m,q)$ (respectively, on $\delta(n,q)$ than $\gamma(n,q)$) is known,
	the upper bounds in Theorem~\ref{main}, {\bf (\ref{gen_upper_bound})},
	can be further improved: use $k = 6\delta(m,q)+1$ (respectively, 	$k = 6\delta(n,q)+1$) in the proof. Similar comments apply to other parts
	of Theorem \ref{main} as well as Theorem \ref{thm_diam_p}.
\end{remark}

\bigskip

\noindent {\bf  (\ref{diam_le_4}).}
Recall the basic fact $\gcd(r,q-1)=1 \Leftrightarrow \{x^r\colon x \in\fq\} = \fq$.


Let $k=4$. If $\gcd(m,q-1) = 1$, a solution to (\ref{eqn:walk_length_k}) of the form $(0,x_2,1,u)$ is seen to exist for any choice of $a,b,u,v\in\fq$.  If $\gcd(n,q-1) = 1$, there exists a solution of the form $(1,x_2,0,u)$. Hence,  $\diam (D_q) \le 4$.

Let $k=3$, and  $\gcd(m,q-1) = \gcd(n,q-1) = 1$.   If $a=0$, then a solution to (\ref{eqn:walk_length_k}) of the form $(x_1,1,u)$  exists. If $a\neq 0$, a solution of the form  $(x_1,0,u)$ exists.  Hence,  $D_q$ is strong and $\diam (D_q) \le 3$. Using the lower bound from part {\bf (\ref{gen_lower_bound})},   we conclude that $\diam (D_q) = 3$.

\bigskip

\noindent {\bf (\ref{main3}).} As was shown in part \ref{diam_le_4}, for any $n$,
$\diam(D(q; 1,n))\le 4$. If, additionally, $\gcd(n,q-1) = 1$, then $\diam(D(q; 1,n)) = 3$.
It turns out that if $p$ does not divide $n$, then only for finitely many $q$ is the diameter of $D(q;1,n)$ actually 4.

For $k=3$, (\ref{eqn:walk_length_k}) is equivalent to
\begin{equation}
\label{eqn:proof_hasse}
(u,v) = (x_3,x_2 x_3^n-x_1 x_2^n + a x_1^n-b),
\end{equation}
which has solution $(x_1,x_2,x_3) = (0,u^{-n}(b+v),u)$, provided $u\neq 0$.

Suppose now that $u = 0$. Aside from the trivial case $a = 0$, the question of the existence of a solution to (\ref{eqn:proof_hasse}) shall be resolved if we prove that the equation
\begin{equation}
\label{eqn:surj}
a x^n - x y^n + c = 0
\end{equation}
 has a solution for any $a, c\in\fq^*$ (for $c=0$, (\ref{eqn:surj}) has solutions).
The projective curve corresponding to this equation is the zero locus of the homogeneous polynomial
\[
 F(X,Y,Z) = aX^n Z - X Y^n + c Z^{n+1}.
\]
It is easy to see that, provided $p$ does not divide $n$,
\[
F=F_X=F_Y=F_Z =0 \;\; \Leftrightarrow \;\; X=Y=Z=0,
\]
and thus the curve has no singularities and is absolutely irreducible.

Counting the two points $[1:0:0]$ and $[0:1:0]$ on the line at infinity $Z = 0$, we obtain from (\ref{hasse_weil_bound}), the inequality
$N\ge q-1-2g\sqrt{q}$, where $N=N(c)$ is the number of solutions of (\ref{eqn:surj}). As $g= n(n-1)/2$,
solving the inequality $q-1-n(n-1)\sqrt{q}>0$ for $q$,  we obtain a lower bound on $q$  for which $N \ge 1$.

\bigskip

\noindent{\bf (\ref{bound_q_le25}a).} 
The result follows from Corollary \ref{thm:diam_le_8} by an argument similar to that of the proof of part {\bf (\ref{gen_upper_bound})}.

\bigskip

\noindent {\bf (\ref{bound_q_m4n4}b).}
For $k=13$, (\ref{eqn:walk_length_k}) is equivalent to
\[
(u,v)
=
(x_{13},
-b + a^m x_1^n -x_1^m x_2^n + x_2^m x_3^n -\dotsb - x_{11}^m x_{12}^n + x_{12}^m x_{13}^n).
\]
If $q > (m-1)^4$, set $x_1 = x_4 = x_7 = x_{10} = 0$,
$x_3 = x_6 = x_9 = x_{12} = 1$. Then
$v - u^n + b = -x_{11}^m + x_8^m - x_5^m + x_2^m$, which has a solution $(x_2,x_5,x_8,x_{11})\in\fq^4$ by Theorem \ref{thm:waring_Small_estimates} and Lemma \ref{lemma:alt}.

\bigskip

\noindent {\bf (\ref{bound_q_le6}c).}
For $k=9$, (\ref{eqn:walk_length_k}) is equivalent to
\[
(u,v)
=
(x_9,
-b + a^n x_1^n -x_1^n x_2^n + x_2^n x_3^n -\dotsb - x_7^m x_8^n + x_8^n x_9^n).
\]
If $q > (n-1)^4$, set $x_1 = x_4 = x_5 = x_8 = 0$,
$x_3 = x_7 = 1$. Then
$v + b = x_2^n + x_6^n$, which has a solution $(x_2,x_6)\in\fq^2$ by Theorem \ref{thm:waring_Small_estimates}.
\bigskip

\section{Proofs of Theorem \ref{thm_diam_p}} \label{proofs2}

\begin{lemma}\label{AutoLemma}
Let $D=D(q;m,n)$.  Then, for any $\lambda\in\mathbb{F}_q^*$, the function $\phi:V(D) \rightarrow V(D)$ given by $\phi((a,b)) = (\lambda a, \lambda^{m+n} b)$ is
a digraph  automorphism of $D$.
\end{lemma}

The proof of the lemma is straightforward. It amounts to showing that $\phi$ is a bijection and that it  preserves adjacency: ${\bf x} \to {\bf y}$ if and only if $\phi({\bf x}) \to \phi({\bf y})$.  We omit the details. Due to Lemma \ref{AutoLemma},  any walk in $D$ initiated at a vertex $(a,b)$ corresponds to a walk initiated at a vertex $(0,b)$ if $a=0$, or at a vertex $(1,b')$,  where $b'= a^{-m-n} b$, if $a\neq 0$. This implies that if we wish to show that $\diam (D_p) \le 2p-1$, it is sufficient to show that the distance from any vertex $(0,b)$ to any other vertex is at most $2p-1$,  and that the distance from any  vertex $(1,b)$ to any other vertex is at most $2p-1$.

First we note that by Theorem \ref{thm_conn}, $D_p = D(p;m,n)$ is strong for any choice of $m,n$.

For $a\in\mathbb{F}_p$, let integer $\overline{a}$,  $0\le \overline{a} \le p-1$, be the  representative of the residue class $a$.


It is easy to check that $\diam (D(2; 1,1)) = 3$.
Therefore, for the remainder of the proof, we may assume that $p$ is odd.
\bigskip

\noindent{\bf (\ref{diam_bound_p}).}
In order to show that diam$(D_p) \le 2p-1$, we use (\ref{eqn:walk_length_k}) with $k= 2p-1$,  and prove that for any two vertices $(a,b)$ and $(u,v)$ of $D_p$ there
is always a solution $(x_1, \ldots, x_{2p-1})\in \fq^{2p-1}$ of
$$(u,v) = (x_{2p-1}, -b + a^mx_1^n - x_1^mx_2^n + x_2^mx_3^n - \dots -
x_{2p-3}^mx_{2p-2}^n + x_{2p-2}^mx_{2p-1}^n),
$$
or, equivalently, a solution ${\bf x} =  (x_1, \ldots, x_{2p-2})\in \fq^{2p-2}$ of
\begin{equation} \label{eq:1}
a^mx_1^n - x_1^mx_2^n + x_2^mx_3^n - \dots -
x_{2p-3}^mx_{2p-2}^n + x_{2p-2}^mu^n = b+v.
\end{equation}
As the upper bound $2p-1$ on the diameter is exact and holds for all $p$, we need a more  subtle  argument compared to the ones we used before.  The only way we can make it is (unfortunately) by  performing  a case analysis on $\overline{b+v}$ with a nested case structure.  In most of the cases we just exhibit  a solution ${\bf x}$ of (\ref{eq:1}) by describing its components $x_i$.
It is always a straightforward verification that ${\bf x}$  satisfies (\ref{eq:1}),  and we will suppress our comments as cases proceed.

Our first observation is that if $\overline{b+v} = 0$, then ${\bf x} = (0,\dots, 0)$ is a  solution to (\ref{eq:1}).
We may assume now that $\overline{b+v}\ne 0$.\\

\noindent\underline{Case 1.1}: $\overline{b+v}\ge \frac{p-1}{2} + 2$

\noindent 
We define the components of ${\bf x}$ as follows:

if $1\le i\le 4(p-(\overline{b+v}))$, then $x_i=0$ for $i\equiv 1,2 \mod{4}$, and $x_i=1$ for $i\equiv 0,3 \mod{4}$;

if $4(p-(\overline{b+v}))< i \le 2p-2$, then $x_i=0$.

Note that $x_i^mx_{i+1}^n = 0$ unless $i\equiv 3 \mod 4$,
in which case $x_i^mx_{i+1}^n = 1$. If we group the terms
in groups of four so that each group is of the form
\[
-x_i^mx_{i+1}^n+x_{i+1}^mx_{i+2}^n-x_{i+2}^mx_{i+3}^n+x_{i+3}^mx_{i+4}^n,
\]
where $i\equiv 1 \mod 4$, then assuming $i$, $i+1$, $i+2$, $i+3$, and $i+4$ are within the range of
$1\le i<i+4 \le 4(\overline{b+v})$, it is easily seen that one group contributes
$-1$ to
\[
a^mx_1^n - x_1^mx_2^n + x_2^mx_3^n - \dots - x_{2p-3}^mx_{2p-2}^n
+ x_{2p-2}^mx_{2p-1}^n.
\]
There are $\frac{4(p-(\overline{b+v}))}{4} = p-(\overline{b+v})$ such
groups, and so the solution provided adds $-1$ exactly
$p-(\overline{b+v})$ times.
Hence, ${\bf x}$ is a solution to (\ref{eq:1}).
\medskip

For the remainder of the proof, solutions to (\ref{eq:1}) will
be given without justification as the justification is similar
to what's been done above.

\vspace{5mm}

\noindent\underline{Case 1.2}: $\overline{b+v}\le \frac{p-1}{2}$

\noindent We define the components of ${\bf x}$ as follows:

if  $1\le i\le 4(\overline{b+v})-1$, then $x_i=0$ for $i\equiv 0,1 \mod{4}$, and $x_i=1$ for $i\equiv 2, 3 \mod{4}$;

if  $4(\overline{b+v})-1< i \le 2p-2 $, then $x_i=0$.

\vspace{5mm}
\noindent\underline{Case 1.3}: $\overline{b+v}= \frac{p-1}{2}+1$

This case requires several nested subcases.

\vspace{3mm}
\underline{Case 1.3.1}: $u=x_{2p-1}=0$

Here, there is no need to restrict $x_{2p-2}$ to be
$0$. The components of a solution ${\bf x}$ of (\ref{eq:1}) are defined as:

if  $1\le i \le 2p-2$, then $x_i=0$ for $i\equiv 1,2 \mod{4}$, and $x_i=1$ for $i\equiv 0,3 \mod{4}$.

\vspace{3mm}

\underline{Case 1.3.2}: $a=0$

Here, there is no need to restrict $x_1$ to be 0. Therefore, the components of a solution ${\bf x}$ of (\ref{eq:1})
are defined as: 

if  $1\le i\le 2p-2$, then $x_i=0$ for $i\equiv 0,3 \mod{4}$, and $x_i=1$ for $i\equiv 1, 2 \mod{4}$.

\vspace{5mm}
\underline{Case 1.3.3}: $u\ne 0$ and $a\ne 0$

Because of Lemma \ref{AutoLemma}, we may assume without loss of generality that $a=1$.
Let $x_{2p-2} = 1$, so that $x_{2p-2}^mu^n=u^n\ne 0$ and let $t=\overline{b+v-u^n}$. Note that
$t\ne\frac{p-1}{2}+1$.

\vspace{3mm}

\underline{Case 1.3.3.1}: $t=0$

The components of a solution ${\bf x}$ of (\ref{eq:1}) are defined as: $x_{2p-2} = 1$, and

if  $1\le i < 2p-2 $, then $x_i=0$.

\vspace{3mm}

\underline{Case 1.3.3.2}: $0< t\le \frac{p-1}{2}$

The components of a solution ${\bf x}$ of (\ref{eq:1}) are defined as: $x_{2p-2} = 1$, and

if  $1\le i\le 4(t-1)+1$, then $x_i=0$ for $i\equiv 2,3 \mod{4}$, and $x_i=1$ for $i\equiv 0,1 \mod{4}$;

if  $4(t-1)+1< i < 2p-2 $, then $x_i=0$.

\vspace{3mm}

\underline{Case 1.3.3.3}: $t\ge \frac{p-1}{2}+2$

The components of a solution ${\bf x}$ of (\ref{eq:1}) are defined as: $x_{2p-2} = 1$, and

if  $1\le i\le 4(p-t)$, then $x_i=0$ for $i\equiv 1,2 \mod{4}$, and $x_i=1$ for $i\equiv 0,3 \mod{4}$;

if  $4(p-t)< i < 2p-2 $, then $x_i=0$.\\

The whole range  of possible values $\overline{b+v}$ has been checked.  Hence, $\diam(D)\le 2p-1$.

%
%
%
%
%
%
%
%
%
%
%

\bigskip

We now show that if $\diam(D)=2p-1$, then $m=n=p-1$. To do so, we assume
that $m\ne p-1$ or $n\ne p-1$ and prove the contrapositive. Specifically, we show that $\diam(D)\le 2p-2<2p-1$ by
again using (\ref{eqn:walk_length_k}) but with $k= 2p-2$.   We prove that for any two vertices $(a,b)$ and $(u,v)$ of $D_p$ there
is always a solution $(x_1, \ldots, x_{2p-2})\in \fq^{2p-2}$ of
$$(u,v) = (x_{2p-2}, b - a^mx_1^n + x_1^mx_2^n - \dots -
x_{2p-4}^mx_{2p-3}^n + x_{2p-3}^mx_{2p-2}^n),
$$
or, equivalently, a solution ${\bf x} =  (x_1, \ldots, x_{2p-3})\in \fq^{2p-3}$ of
\begin{equation} \label{eq:2}
-a^mx_1^n + x_1^mx_2^n - x_2^mx_3^n + \dots -
x_{2p-4}^mx_{2p-3}^n + x_{2p-3}^mu^n = -b+v.
\end{equation}
We perform a case analysis on $\overline{-b+v}$.

\vspace{5mm}
Our first observation is that if $\overline{-b+v} = 0$, then ${\bf x} = (0,\dots, 0)$ is a  solution to (\ref{eq:2}). We may
assume for the remainder of the proof that $\overline{-b+v}\ne 0$.

\vspace{3mm}

\noindent\underline{Case 2.1}: $\overline{-b+v}\le \frac{p-1}{2}-1$

\noindent We define the components of ${\bf x}$ as follows:

if  $1\le i\le 4(\overline{-b+v})$, then $x_i=0$ for $i\equiv 1,2 \mod{4}$, and $x_i=1$ for $i\equiv 0, 3 \mod{4}$;

if  $4(\overline{-b+v})< i \le 2p-3 $, then $x_i=0$.

\vspace{3mm}

\noindent\underline{Case 2.2}: $\overline{-b+v}\ge \frac{p-1}{2}+2$

\noindent We define the components of ${\bf x}$ as follows:

if  $1\le i\le 4(p-(\overline{-b+v}))-1$, then $x_i=0$ for $i\equiv 0,1 \mod{4}$, and $x_i=1$ for $i\equiv 2, 3 \mod{4}$;

if  $4(p-(\overline{-b+v}))-1< i \le 2p-3 $, then $x_i=0$.

\vspace{3mm}

\noindent\underline{Case 2.3}: $\overline{-b+v}= \frac{p-1}{2}$

\underline{Case 2.3.1}: $a=0$

We define the components of ${\bf x}$ as:

if  $1\le i\le 2p-3$, then $x_i=0$ for $i\equiv 0,3 \mod{4}$, and $x_i=1$ for $i\equiv 1, 2 \mod{4}$.

\vspace{3mm}

\underline{Case 2.3.2}: $a\ne 0$

Here, we may assume without loss of generality that $a=1$ by Lemma (\ref{AutoLemma}).

\vspace{3mm}

\underline{Case 2.3.2.1}: $n\ne p-1$

If $n\ne p-1$, then there exists $\beta\in\mathbb{F}_p^*$ such that $\beta^n\not\in\{0,1\}$. For such a $\beta$,
let $x_1=\beta$ and consider $t=\overline{-b+v+a^mx_1^n}=\overline{-b+v+\beta^n}\not\in\{\frac{p-1}{2}, \frac{p-1}{2}+1 \}$.

\vspace{3mm}

\underline{Case 2.3.2.1.1}: $t=0$

\noindent\noindent We define the components of ${\bf x}$ as: $x_1=\beta$ and

if  $2\le i \le 2p-3 $, then $x_i=0$.

\vspace{3mm}

\underline{Case 2.3.2.1.2}: $t\le \frac{p-1}{2}-1$

\noindent\noindent We define the components of ${\bf x}$ as: $x_1=\beta$ and

if  $2\le i\le 4t$, then $x_i=0$ for $i\equiv 1,2 \mod{4}$, and $x_i=1$ for $i\equiv 0, 3 \mod{4}$;

if  $4t< i \le 2p-3 $, then $x_i=0$.

\vspace{3mm}

\underline{Case 2.3.2.1.3}: $t\ge \frac{p-1}{2}+2$

\noindent We define the components of ${\bf x}$ as: $x_1=\beta$ and

if  $2\le i\le 4(p-t)+1$, then $x_i=0$ for $i\equiv 2,3 \mod{4}$, and $x_i=1$ for $i\equiv 0, 1 \mod{4}$;

if  $4(p-t)+1< i \le 2p-3 $, then $x_i=0$.

\vspace{3mm}

\underline{Case 2.3.2.2}: $n=p-1$

\vspace{3mm}

\underline{Case 2.3.2.2.1}: $u\in\mathbb{F}_p^*$

Here, we have that $u^n=1$, so that the components of a solution ${\bf x}$ of (\ref{eq:2}) are defined as:

if  $1\le i\le 2p-3$, then $x_i=0$ for $i\equiv 1,2 \mod{4}$, and $x_i=1$ for $i\equiv 0, 3 \mod{4}$.

\vspace{3mm}

\underline{Case 2.3.2.2.2}: $u=0$

Since $n=p-1$, it must be the case that $m\ne p-1$ so that there exists $\alpha\in\mathbb{F}_p^*$ such that $\alpha^m\not\in\{0.1 \}$.
For such an $\alpha$, let $x_2=\alpha, x_3=1$ and consider $t=\overline{-b+v+x_2^mx_3^n}=\overline{-b+v+\alpha^m}
\not\in\{\frac{p-1}{2}, \frac{p-1}{2}+1 \}$.

\vspace{3mm}

\underline{Case 2.3.2.2.2.1}: $t=0$

\noindent We define the components of ${\bf x}$ as: $x_1=0, x_2=\alpha, x_3=1$ and

if  $4 \le  i \le 2p-3 $, then $x_i=0$.

\vspace{3mm}

\underline{Case 2.3.2.2.2.2}: $t\le \frac{p-1}{2}-1$

\noindent We define the components of ${\bf x}$ as: $x_1=0, x_2=\alpha, x_3=1$ and

if  $4\le i\le 4t$, then $x_i=0$ for $i\equiv 1,2 \mod{4}$, and $x_i=1$ for $i\equiv 0, 3 \mod{4}$;

if  $4t< i \le 2p-3 $, then $x_i=0$.

\vspace{3mm}

\underline{Case 2.3.2.2.2.3}: $t\ge \frac{p-1}{2}+2$

\noindent We define the components of ${\bf x}$ as: $x_1=0, x_2=\alpha, x_3=1$ and

if  $4\le i\le 4(p-t)+3$, then $x_i=0$ for $i\equiv 0,1 \mod{4}$, and $x_i=1$ for $i\equiv 2, 3 \mod{4}$;

if  $4(p-t)+3< i \le 2p-3 $, then $x_i=0$.

\vspace{3mm}

\noindent\underline{Case 2.4}: $\overline{-b+v}= \frac{p-1}{2}+1$

\vspace{3mm}

\underline{Case 2.4.1}: $u=0$

We define the components of ${\bf x}$ as:

if  $1\le i\le 2p-3$, then $x_i=0$ for $i\equiv 0,1 \mod{4}$, and $x_i=1$ for $i\equiv 2, 3 \mod{4}$.

\vspace{3mm}

\underline{Case 2.4.2}: $u\ne 0$

Here, we may assume without loss of generality that $u=1$ by Lemma (\ref{AutoLemma}).

\vspace{3mm}

\underline{Case 2.4.2.1}: $m\ne p-1$

If $m\ne p-1$, then there exists $\alpha\in\mathbb{F}_p^*$ such that $\alpha^m\not\in\{0,1\}$. For such an $\alpha$,
let $x_{2p-3}=\alpha$ and consider $t=\overline{-b+v-x_{2p-3}^mu^n}=\overline{-b+v-\alpha^m}\not\in\{\frac{p-1}{2}, \frac{p-1}{2}+1 \}$.

\vspace{3mm}

\underline{Case 2.4.2.1.1}: $t=0$

\noindent We define the components of ${\bf x}$ as: $x_{2p-3}=\alpha$ and

if  $1 \le i \le 2p-4 $, then $x_i=0$.

\vspace{3mm}

\underline{Case 2.4.2.1.2}: $t\le \frac{p-1}{2}-1$

\noindent We define the components of ${\bf x}$ as: $x_{2p-3}=\alpha$ and

if  $1\le i\le 4t$, then $x_i=0$ for $i\equiv 1,2 \mod{4}$, and $x_i=1$ for $i\equiv 0, 3 \mod{4}$;

if  $4t< i \le 2p-4 $, then $x_i=0$.

\vspace{3mm}

\underline{Case 2.4.2.1.3}: $t\ge \frac{p-1}{2}+2$

\noindent We define the components of ${\bf x}$ as: $x_{2p-3}=\alpha$ and

if  $1\le i\le 4(p-t)-1$, then $x_i=0$ for $i\equiv 0,1 \mod{4}$, and $x_i=1$ for $i\equiv 2, 3 \mod{4}$;

if  $4(p-t)-1< i \le 2p-4 $, then $x_i=0$.

\vspace{3mm}

\underline{Case 2.4.2.2}: $m=p-1$

\vspace{3mm}

\underline{Case 2.4.2.2.1}: $a\in\mathbb{F}_p^*$

Here, we have that $a^m=1$, so that the components of a solution ${\bf x}$ of (\ref{eq:2}) are defined as:

if  $1\le i\le 2p-5$, then $x_i=0$ for $i\equiv 2,3 \mod{4}$, and $x_i=1$ for $i\equiv 0, 1 \mod{4}$.

\vspace{3mm}

\underline{Case 2.4.2.2.2}: $a=0$

Since $m=p-1$, it must be the case that $n\ne p-1$ so that there exists $\beta\in\mathbb{F}_p^*$ such that $\beta^n\not\in\{0.1 \}$.
For such a $\beta$, let $x_{2p-5}=1, x_{2p-4}=\beta$ and consider $t=\overline{-b+v-x_{2p-5}^mx_{2p-4}^n}=\overline{-b+v-\beta^n}
\not\in\{\frac{p-1}{2}, \frac{p-1}{2}+1 \}$.

\vspace{3mm}

\underline{Case 2.4.2.2.2.1}: $t=0$

\noindent We define the components of ${\bf x}$ as: $x_{2p-5}=1, x_{2p-4}=\beta, x_{2p-3}=0$ and

if  $1\le i \le 2p-6 $, then $x_i=0$.

\vspace{3mm}

\underline{Case 2.4.2.2.2.2}: $t\le \frac{p-1}{2}-1$

\noindent We define the components of ${\bf x}$ as: $x_{2p-5}=1, x_{2p-4}=\beta, x_{2p-3}=0$ and

if  $1\le i\le 4t-2$, then $x_i=0$ for $i\equiv 0,3 \mod{4}$, and $x_i=1$ for $i\equiv 1, 2 \mod{4}$;

if  $4t-2< i \le 2p-6 $, then $x_i=0$.

\vspace{3mm}

\underline{Case 2.4.2.2.2.3}: $t\ge \frac{p-1}{2}+2$

\noindent We define the components of ${\bf x}$ as: $x_{2p-5}=1, x_{2p-4}=\beta, x_{2p-3}=0$ and

if  $1\le i\le 4(p-t)-1$, then $x_i=0$ for $i\equiv 0,1 \mod{4}$, and $x_i=1$ for $i\equiv 2, 3 \mod{4}$;

if  $4(p-t)-1< i \le 2p-6 $, then $x_i=0$.\\

All cases have been checked, so if $m\ne p-1$ or $n\ne p-1$, then $\diam(D) < 2p-1$.

\vspace{5mm}

We now prove that if $m=n=p-1$, then $d:= \diam (D(p;m,n))=2p-~1$.
In order to do this, we explicitly describe the structure of the digraph $D(p;p-1,p-1)$,
from which the diameter becomes clear. In this description, we
look at sets of vertices of a given distance from the vertex $(0,0)$, and show that some of them are at distance $2p-1$.
We recall the following important general properties of our digraphs that will be used in the proof.
\begin{itemize}
\item Every out-neighbor $(u,v)$ of a vertex $(a,b)$ of $D(q;m,n)$ is completely determined by its first component $u$.

\item Every vertex of $D(q;m,n)$ has its out-degree and in-degree  equal $q$.
\item In $D(q; m,m)$,   ${\bf x}\to {\bf y}$ if and only if
${\bf y}\to {\bf x}$
\end{itemize}

In $D(p;p-1,p-1)$, we have that $(x_1, y_1)\to
(x_2, y_2)$ if and only if
\[
y_1 + y_2 = x_1^{p-1}x_2^{p-1} = \begin{cases}
      0 & \textrm{ if $x_1=0$ or $x_2=0$}, \\
      1 & \textrm{ if $x_1$ and $x_2$ are non-zero}. \\
   \end{cases}
\]
For notational convenience, we set
\[
(*, a) = \{(x, a): x\in\mathbb{F}_p^*\}
\]
and, for $1\le k\le d$, let
\[
N_k = \{v\in V(D(p;m,n)): \text{dist}((0,0), v) = k \}.
\]
We assume that $N_0=\{(0,0)\}$.
It is clear from this definition that these $d+1$ sets  $N_k$ partition the vertex set of $D(p;p-1,p-1)$;  for every $k$, $1\le k\le d-1$,  every out-neighbor of a vertex from $N_k$ belongs  to $N_{k-1}\cup N_k\cup N_{k+1}$,  and $N_{k+1}$ is the set of all out-neighbors of all vertices from
$N_k$ which are not in $N_{k-1}\cup N_k$.

Thus we have $N_0=\{(0,0)\}$, $N_1= (*,0)$, $N_2=(*,1)$, $N_3=\{(0,-1)\}$. If $p>2$, $N_4=\{(0,1)\}$, $N_5=(*,-1)$.  As there exist two (opposite) arcs between each vertex of $(*,x)$ and each vertex $(*,-x+1)$,  these  subsets of vertices induce the complete bipartite subdigraph $\overrightarrow{K}_{p-1,p-1}$ if $x\ne -x+1$, and the complete subdigraph $\overrightarrow{K}_{p-1}$ if $x =-x+1$. Note that our $\overrightarrow{K}_{p-1,p-1}$ has no loops, but $\overrightarrow{K}_{p-1}$  has a loop on every vertex.
Digraph $D(5;4,4)$ is depicted in Fig. $1.2$.

\begin{figure}
\begin{center}
\begin{tikzpicture}

\tikzset{vertex/.style = {shape=circle,draw,inner sep=2pt,minimum size=.5em, scale = 1.0},font=\sffamily\scriptsize\bfseries}
\tikzset{edge/.style = {->,> = stealth'},shorten >=1pt}

\node[vertex,label={[xshift=-0.2cm, yshift=0.0cm]$(0,0)$}] (a) at  (0,0) {};

\node[vertex] (b1) at  (1,1.5) {};
\node[vertex] (b2) at  (1,.5) {};
\node[vertex] (b3) at  (1,-.5) {};
\node[vertex,label={[xshift=0.0cm, yshift=-0.8cm]$(\ast,0)$}] (b4) at  (1,-1.5) {};

\node[vertex] (c1) at  (2,1.5) {};
\node[vertex] (c2) at  (2,.5) {};
\node[vertex] (c3) at  (2,-.5) {};
\node[vertex,label={[xshift=0.0cm, yshift=-0.8cm]$(\ast,1)$}] (c4) at  (2,-1.5) {};

\node[vertex,label={[xshift=0.25cm, yshift=-0.8cm]$(0,-1)$}] (d) at  (3,0) {};

\node[vertex,label={[xshift=-0.2cm, yshift=0.0cm]$(0,1)$}] (e) at  (4,0) {};

\node[vertex] (f1) at  (5,1.5) {};
\node[vertex] (f2) at  (5,.5) {};
\node[vertex] (f3) at  (5,-.5) {};
\node[vertex,label={[xshift=0.0cm, yshift=-0.8cm]$(\ast,-1)$}] (f4) at  (5,-1.5) {};

\node[vertex] (g1) at  (6,1.5) {};
\node[vertex] (g2) at  (6,.5) {};
\node[vertex] (g3) at  (6,-.5) {};
\node[vertex,label={[xshift=0.0cm, yshift=-0.8cm]$(\ast,2)$}] (g4) at  (6,-1.5) {};

\node[vertex,label={[xshift=0.25cm, yshift=-0.8cm]$(0,-2)$}] (h) at  (7,0) {};

\node[vertex,label={[xshift=-0.3cm, yshift=0.00cm]$(0,2)$}] (i) at  (8,0) {};

\node[vertex] (j1) at  (9,1.5) {};
\node[vertex] (j2) at  (9,.5) {};
\node[vertex] (j3) at  (9,-.5) {};
\node[vertex,label={[xshift=0.0cm, yshift=-0.8cm]$(\ast,-2)$}] (j4) at  (9,-1.5) {};

\path
    (a) edge [->,>={stealth'[flex,sep=-1pt]},loop,out=240,in=270, looseness = 50] node {} (a);

\foreach \x in {b1,b2,b3,b4}
{
    \draw [edge] (a) to (\x);
    \draw [edge] (\x) to (a);
}

\foreach \x in {b1,b2,b3,b4}
{
    \foreach \y in {c1,c2,c3,c4}
    {
            \draw [edge] (\x) to (\y);
            \draw [edge] (\y) to (\x);
    }
}

\foreach \x in {c1,c2,c3,c4}
{
    \draw [edge] (d) to (\x);
    \draw [edge] (\x) to (d);
}

\draw [edge] (d) to (e);
\draw [edge] (e) to (d);

\foreach \x in {f1,f2,f3,f4}
{
    \draw [edge] (e) to (\x);
    \draw [edge] (\x) to (e);
}

\foreach \x in {f1,f2,f3,f4}
{
    \foreach \y in {g1,g2,g3,g4}
    {
            \draw [edge] (\x) to (\y);
            \draw [edge] (\y) to (\x);
    }
}

\foreach \x in {g1,g2,g3,g4}
{
    \draw [edge] (h) to (\x);
    \draw [edge] (\x) to (h);
}

\draw [edge] (h) to (i);
\draw [edge] (i) to (h);

\foreach \x in {j1,j2,j3,j4}
{
    \draw [edge] (i) to (\x);
    \draw [edge] (\x) to (i);
}

\path
    (j1) edge [->,>={stealth'[flex,sep=-1pt]},loop,out=30,in=-20, looseness = 35] node {} (j1);
\path
    (j2) edge [->,>={stealth'[flex,sep=-1pt]},loop,out=30,in=-20, looseness = 35] node {} (j2);
\path
    (j3) edge [->,>={stealth'[flex,sep=-1pt]},loop,out=30,in=-20, looseness = 35] node {} (j3);
\path
    (j4) edge [->,>={stealth'[flex,sep=-1pt]},loop,out=30,in=-20, looseness = 35] node {} (j4);

\path
(j1) edge[bend right,<->,>=stealth'] node [left] {} (j2);
\path
(j1) edge[bend right = 60,<->,>=stealth'] node [left] {} (j3);
\path
(j1) edge[bend right = 320,<->,>=stealth'] node [left] {} (j4);
\path
(j2) edge[bend right,<->,>=stealth'] node [left] {} (j3);
\path
(j2) edge[bend right = 60,<->,>=stealth'] node [left] {} (j4);
\path
(j3) edge[bend right,<->,>=stealth'] node [left] {} (j4);
\end{tikzpicture}
\caption{The digraph $D(5;4,4)$: $x_2+y_2 = x_1^4y_1^4$.}
\end{center}
\end{figure}
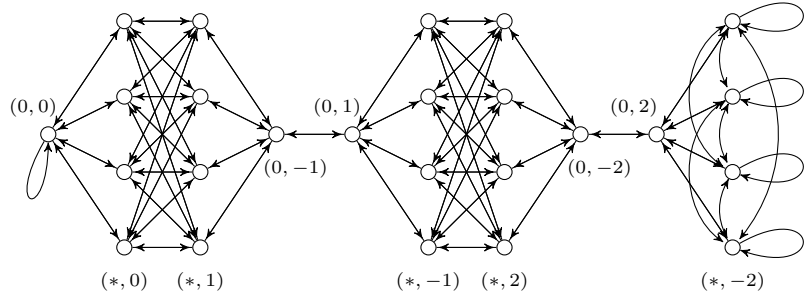

The structure of $D(p;p-1,p-1)$ for any other prime $p$ is similar. We can describe it as follows:    for each $t\in \{0,1, \ldots, (p-1)/2\}$, let
$$
N_{4{\overline t}} = \{(0, t)\}, \;\;
N_{4{\overline t}+1} = (*, -t),
$$
and for each $t\in \{0,1, \ldots, (p-3)/2\}$, let
$$
N_{4{\overline t}+2} = (*, t+1), \;
N_{4{\overline t}+3} = \{(0, -t-1)\}.
$$
Note that for $0\le {\overline t}<(p-1)/2$, $N_{4{\overline t}+1}\neq N_{4{\overline t}+2}$,  and for ${\overline t}=(p-1)/2$, $N_{2p-1}=(*,(p+1)/2)$.  Therefore, for $p\ge 3$,  $D(p;p-1,p-1)$ contains $(p-1)/2$ induced copies of
$\overrightarrow{K}_{p-1,p-1}$ with partitions  $N_{4{\overline t}+1}$ and $N_{4{\overline t}+2}$,  and a copy of $\overrightarrow{K}_{p-1}$  induced by $N_{2p-1}$. The proof is a trivial induction on $\overline{t}$. Hence, $\diam (D(p;p-1,p-1)) = 2p-1$.  This ends the proof  of Theorem~\ref{thm_diam_p}~(\ref{diam_bound_p}).

\bigskip

\noindent{\bf (\ref{bound_p_sqrt60}).}
We follow the argument of the proof of Theorem \ref{main}, part {\bf (\ref{gen_upper_bound})} and use Lemma \ref{lemma:alt}, with $k = 6\delta(m,p)+1$. We note, additionally, that if $m\not\in\{p,(p-1)/2\}$, then $\gcd(m,p-1) < (p-1)/2$, which implies $|\{ x^m \colon x\in\mathbb{F}_p^\ast \} | > 2$. The result then follows from Theorem \ref{thm:Cochrane_Pinner}.

\bigskip

\noindent{\bf (\ref{bound_p_le10}).}
We follow the argument of the proof of Theorem \ref{main}, part {\bf (\ref{bound_q_m4n4}b)} and use Lemma \ref{lemma:alt} and Theorem \ref{thm:waring_small_estimates}.
\medskip

This ends the proof of  Theorem~\ref{thm_diam_p}.

\bigskip

\section{Concluding remarks.}\label{open}

Many  results in this paper follow the same pattern: if Waring's number $\delta(r,q)$ exists and is bounded above by $\delta$, then one can show that $\diam(D(q;m,n))\le 6\delta + 1$. Determining the exact value of  $\delta(r,q)$ is an open  problem,  and it is likely to be very hard. Also,  the upper bound $6\delta +1$ is not exact in general. Out of all partial results concerning  $\delta(r,q)$, we used only those ones which helped us deal with the cases of the diameter of $D(q; m,n)$ that we considered, especially where the diameter was small.  We left out applications of all asymptotic bounds on  $\delta(r,q)$. Our computer work demonstrates  that some upper bounds on the diameter  mentioned in this paper are still far from being tight.  Here we wish to mention only a few strong patterns that we observed but have not been able to prove so far. We state them as problems.
\bigskip

\noindent{\bf Problem 1.}
 Let $p$ be prime, $q=p^e$, $e \ge 2$, and suppose $D(q;m,n)$ is strong. Let
 $r$ be the largest divisor of $q-1$
 not divisible by any
 $q_d = (p^e-1)/(q^d-1)$
 where $d$ is a positive divisor of $e$ smaller than $e$. Is it true that
\[
\max_{1\le m\le n\le q-1}
\{
\diam(D(q;m,n))
\}
=
\diam(D(q;r,r))?
\]
Find an upper bound on $\diam(D(q;r,r))$ better than the one of
Theorem \ref{main}, part {\bf (\ref{bound_q_le6}c)}.
\bigskip

\noindent{\bf Problem 2.} Is it true that for every prime $p$ and  $1\le m \le n$,
$(m,n)\neq (p-1,p-1)$, $\diam (D(p;m,n)) \le (p+3)/2$ with the equality  if and only if $(m,n)=((p-1)/2, (p-1)/2)$ or $(m,n)=((p-1)/2, p-1)$?

\bigskip

\noindent{\bf Problem 3.}  Is it true that  for every prime $p$,  $\diam (D(p;m,n))$ takes only  one of two consecutive values which are completely determined by $\gcd((p-1, m, n)$?

\section{Acknowledgement}
The authors are thankful to the anonymous referee whose careful reading and thoughtful comments led to a number of significant improvements in the paper.


\begin{thebibliography}{99}

\bibitem{Bang_Jensen_Gutin}
    J.~Bang-Jensen, G.~Gutin,
    \underbar{Digraphs: Theory, Algorithms and Applications},
    Springer 2009.

\bibitem{Bhashkaran_1966}
     M.~Bhaskaran,
     {Sums  of  $m$-th  powers  in  algebraic  and  abelian  number  fields},
     Arch.  Math. (Basel) 17 (1966), 497-504; Correction, ibid. 22 (1972), 370-371.


\bibitem{CLL14} S.M. Cioab\u{a}, F. Lazebnik and W. Li,
On the Spectrum of Wenger Graphs,
J. Combin. Theory Ser. B 107: (2014), 132--139.

\bibitem{Cipra_thesis}
    J.~Cipra,
    Waring's number in finite fields,
    Doctoral Thesis,
    Kansas State University, 2010.

\bibitem{Cipra_2009}
    J.~Cipra.
    Waring's number in a finite field,
    Integers 8 2009.


\bibitem{Cochrane_Pinner_2008}
    T.~Cochrane, C.~Pinner,
    Sum-product estimates applied to Waring's problem mod $p$,
    Integers 8 (2008), A46.

\bibitem{DLV05} V. Dmytrenko, F. Lazebnik and R. Viglione,
An Isomorphism Criterion for Monomial Graphs,
J. Graph Theory 48 (2005), 322--328.

\bibitem{DLW07}  V. Dmytrenko, F. Lazebnik and J. Williford,
On monomial graphs of girth eight,
Finite Fields Appl.  13 (2007), 828--842.


\bibitem{Hirschfeld}
    {J.W.P.~Hirschfeld, G.~Korchm\'{a}ros, F.~Torres},
    \ul{Algebraic Curves over a Finite Field},
    Princeton Series in Applied Mathematics, 2008.

\bibitem{Hua_Vandiver}
    L.K.~Hua, H.S.~Vandiver,
    Characters over certain types of rings with applications
    to the theory of equations in a finite field,
    {Proc. Natl. Acad. Sci. U.S.A.} 35 (1949), 94--99.


\bibitem{Kod14} A. Kodess,
    Properties of some algebraically defined digraphs,
    Doctoral Thesis,
    University of Delaware, 2014.

\bibitem{Kod_Laz_15} A. Kodess, F. Lazebnik,
    Connectivity of some algebraically defined digraphs,
    {Electron.\ J.\ Combin.},
    22(3) (2015), \#P3.27, 1--11.


\bibitem{Kron12} B.G. Kronenthal, Monomial graphs and generalized quadrangles,
{Finite Fields Appl.} 18 (2012), 674--684.


\bibitem{Lazebnik_Mubayi}
    F.~Lazebnik, D.~Mubayi,
    New lower bounds for Ramsey numbers of graphs
    and hypergraphs,
    Adv. Appl. Math. 8 (3/4) (2002), 544--559.


\bibitem{Lazebnik_Thomason}
    F.~Lazebnik, A.~Thomason,
    Orthomorphisms and the construction of
    projective planes,
    {Math.\ Comp.} 73 (247) (2004), 1547--1557.


\bibitem{Lazebnik_Verstraete}
     F.~ Lazebnik, J.~Verstra\"{e}te,
     On hypergraphs of girth five,
     {Electron.\ J.\ Combin.} 10 (R25) (2003), 1--15.

\bibitem{Lazebnik_Viglione}
    F.~Lazebnik, R.~Viglione,
    An infinite series of regular edge-
    but not vertex-transitive graphs,
    {J. Graph Theory} 41 (2002), 249--258.


\bibitem{LazWol01}
    F.~Lazebnik, A.J.~Woldar,
    General properties of some families of graphs defined by systems of equations,
    {J.\ Graph Theory} 38 (2) (2001), 65--86.


\bibitem{Handbook2013}
G.L. Mullen and D. Panario,
\underline{Handbook of Finite Fields}, CRC Press, Taylor $\&$ Francis Group,  2013.


\bibitem{Small_1977}
     C.~Small,
     {Sums of powers in large fields},
     Proc. Amer. Math. Soc. 65 (1977), p. 35–35.

\bibitem{Szonyi1997} T. Sz\H{o}nyi,  Some applications of algebraic curves in finite geometry and combinatorics. In  \underbar{Surveys in Combinatorics}, Edited by R.A. Bailey, pp. 197--236, 1997 (London), vol. 241 of London Math. Soc. Lecture Note Ser.,  Cambridge Univ. Press,
    Cambridge, 1997.


\bibitem{TerWil12} T.A. Terlep, J. Williford, Graphs from generalized Kac-Moody algebras,
SIAM J. Discrete Math. 26 no. 3 (2012), 1112--1120.


\bibitem{Ust07} V.A. Ustimenko, On the extremal regular directed graphs without
commutative diagrams and their applications in coding theory and cryptography,
\textsl{Albanian J. Math.}   1 (01/2007), 283--295.

\bibitem{Viglione_thesis} R.~Viglione,
    Properties of some algebraically defined graphs,
    Doctoral Thesis,
    University of Delaware, 2002.

\bibitem{VigDiam08} R.~Viglione, On Diameter of Wenger Graphs,
 Acta Appl. Math. 104 (2) (11/2008), 173--176.

\bibitem{Weil}
    A.~Weil,
    Number of solutions of equations in finite fields,
    Bull.\ Amer.\ Math. Soc. 55 (1949), 497--508.
\end{thebibliography}
\end{document}